\documentclass[a4paper,12pt]{amsart}

\usepackage{amsfonts}
\usepackage{amsmath}
\usepackage{amssymb}
\usepackage{graphicx}

\usepackage[usenames]{color}
\usepackage[colorlinks]{hyperref}

\newtheorem{thm}{Theorem}[section]
\newtheorem{lem}[thm]{Lemma}

\theoremstyle{definition}

\theoremstyle{remark}
\newtheorem{rem}{Remark}[section]

\numberwithin{equation}{section}

\def\d{\mathrm d}

\begin{document}

\title[On the $q$-analogue of Duhamel's principle]{On the $q$-analogue of Duhamel's principle}

\author[M. Sebih]{Mohammed Elamine Sebih}
\address{
  Mohammed Elamine Sebih:
  \endgraf
  Laboratory of Geomatics, Ecology and Environment (LGEO2E)
  \endgraf
  University Mustapha Stambouli of Mascara, 29000 Mascara
  \endgraf
  Algeendgraf
  {\it E-mail address} {\rm sebihmed@gmail.com, ma.sebih@univ-mascara.dz}
}

\author[S. Shaimardan]{Serikbol Shaimardan}
\address{
  Serikbol Shaimardan:
  \endgraf 
  L. N. Gumilyov Eurasian National University,
  \endgraf
  Astana, Kazakhstan
  \endgraf
  and
  \endgraf   
  Department of Mathematics: Analysis, Logic and Discrete Mathematics
  \endgraf
  Ghent University, Krijgslaan 281, Building S8, B 9000 Ghent
  \endgraf
  Belgium
  \endgraf  
  {\it E-mail address:} {\rm shaimardan.serik@gmail.com}
  }

  \author[I. Ali]{Irfan Ali}
\address{
  Irfan Ali:
  \endgraf
  Department of Mathematics: Analysis, Logic and Discrete Mathematics
  \endgraf
  Ghent University, Krijgslaan 281, Building S8, B 9000 Ghent
  \endgraf
  Belgium
  \endgraf
  and
  \endgraf
  Department of Mathematical Sciences,
  \endgraf
  BUITEMS, Airport Road Quetta
  \endgraf
  Pakistan
  \endgraf
  {\it E-mail address} {\rm irfan.ali@ugent.be, irfan.changazi1@gmail.com}
}

\thanks{This research was funded by the FWO Odysseus 1 grant G.0H94.18N: Analysis and Partial Differential Equations, and by the Methusalem programme of the Ghent University Special Research Fund (BOF) (Grant number 01M01021).}


\begin{abstract}
In this paper, we revisit the classical Duhamel's principle and provide a
self-contained proof of this fundamental tool for linear evolution equations and systems of coupled equations. Moreover, we establish a $q$-analogue of Duhamel's
principle for $q$-evolution equations of order $k\geq 1$ generated by
Jackson's $q$-difference operator.
\end{abstract}

\maketitle

%%%%%%%%%%%%%%%%%%%%%%%%%%%%%%%%%%%%%%%%%%%%%%%%%%%%%%%%%%%%%%%%%%%%%%%%%%%%%%%%%%%%%%%%%%%%%%%%%%%%%%%%%%%%%%%%%%%%%%%%%%%%%%%%%%%%%%%%%%%%%%%%%%%%%%%%%%%%%%%%%%%%%%%%%%%%%%%%%%%%%%%%%%%%%%%%%%%%

\section{Introduction}
The theory of $q$-calculus, also referred to as quantum calculus, provides a generalisation 
of classical calculus in which the notion of limit is replaced by a deformation 
parameter $q$. Within this framework, the classical notions of differentiation and integration are replaced 
by the $q$-derivative and the Jackson $q$-integral, respectively. Both operators 
depend on the parameter $q$ and reduce to their classical counterparts as $q \to 1$. Consequently, $q$-calculus can be viewed as a deformation of classical 
analysis that preserves many of its fundamental structures while introducing 
discrete features \cite{Ern12,GR04}.

The foundations of $q$-calculus were established by F. H. Jackson, who introduced 
the Jackson $q$-integral and developed many basic results concerning 
$q$-difference operators \cite{Jac10}. Since then, the theory has developed into
a powerful tool with applications in various areas of mathematics, including
basic hypergeometric functions, orthogonal polynomials, combinatorics, and
mathematical physics \cite{GR04,Ism05}. In particular, $q$-calculus 
plays a central role in the study of $q$-difference equations, which naturally arise 
in discrete dynamical systems and in certain models of quantum theory.

Over the past decades, many authors have investigated qualitative and 
analytical properties of $q$-difference equations, including existence, 
oscillation, and stability of solutions. Recent works have addressed various 
classes of $q$-difference and fractional $q$-difference equations, showing 
the wide applicability of $q$-calculus techniques in modern mathematical 
analysis \cite{AHH24,ZG24,SSK24,KR24}. These developments highlight the 
importance of extending classical analytical tools to the $q$-setting.

Among the fundamental analytical techniques used in the study of partial differential equations is Duhamel's principle, which is a key in proving existence and uniqueness of solutions to non-homogeneous problems. Numerous equations, including the wave equation, the heat equation and the Scr\"odinger equation are studied using this principle \cite{ER18,Eva98,Seb22}. It is also an indispensable tool in the study of nonlinear partial differential equations where one treats the nonlinearity as an inhomogeneity. It provides a representation formula for 
the solution of nonhomogeneous linear differential equations. In the classical 
setting, this principle expresses the solution of a forced equation in terms of 
the solution of the corresponding homogeneous problem together with an integral involving the forcing term.

Motivated by the increasing interest in $q$-difference equations and quantum
analogues of classical models, it is natural to investigate whether Duhamel's
principle admits a suitable formulation within the framework of $q$-calculus.
Such an extension is particularly relevant in the study of $q$-evolution
equations, where classical differential operators are replaced by
$q$-difference operators and the classical integral is replaced by the
Jackson $q$-integral.

Surprisingly, a carefully written proof of the classical Duhamel's principle seems to be missing in the literature. One aim in this paper is therefore to fill this gap and to provide a clear proof of this classical result. The main objective however is to establish a $q$-analogue of 
Duhamel's principle for a class of $q$-differential equations generated by Jackson's operator. This result provides a useful analytical framework for studying forced $q$-evolution equations and contributes to the ongoing development of 
analytical methods in quantum calculus.

The paper is organised as follows: In the next section we give the necessary preliminaries about quantum calculus. For convenience of the reader, Duhamel's principle in the classical case is included in Section 3. In Section 4, we provide the proofs for analogous versions of Duhamel's principle for systems of coupled equations of orders $1$ and $2$, as well as for mixed-order systems. Section 5 contains the main results of the paper together with their
proofs.

\section{Preliminaries}
In this section, we recall some  notations  related to the  $q$-calculus. We will always assume that $0<q<1$. The $q$-real number $[\alpha ]_q$ is defined by
\begin{equation}
[\alpha ]_{q}:=\frac{1-q^{\alpha }}{1-q}.    \end{equation}
We define the $q$-lattice $\mathbb{R}_q \subset \mathbb{R}$ by
\begin{equation*}
    \mathbb{R}_q := \{ x^k : x=\pm q^k; ~ k\in \mathbb{Z}\},
\end{equation*}
and accordingly, for $a,b \in \mathbb{R}$,
\begin{equation*}
    [a,b]_q := [a,b]\cap \mathbb{R}_q.
\end{equation*}
The $q$-analogue differential operator $D_{q}f(x)$ (also called Jackson's operator) is defined by 
\begin{equation}\label{Jack}
D_{q}f(x)=\frac{f(x)-f(qx)}{x(1-q)}.
\end{equation}
The $q$-derivative of a product of two functions has the form
\begin{equation*}
D_{q}(fg)(x)= f(qx)D_{q}(g)(x)+D_{q}(f)(x)g(x),
\end{equation*}
and the $q^2$-derivative (called Rubin's operator) is defined by (see \cite{Rubin1} and \cite{Rubin2})
\begin{equation}\label{Rub}
\partial_{q}f(x)=\frac{f(q^{-1}x)+f(-q^{-1}x)
-f(qx)+f(-qx)-2f(-x)}{2x(1-q)}.
\end{equation}
Jackson's and Rubin's operators are related by
\begin{equation*}
   \partial_{q}f(x) =  D_{q}f(x) + D_{q^{-1}}f(x).
\end{equation*}
The $q$-integral (or Jackson integral) is defined by (see \cite{Jac10})
\begin{equation}
\int\limits_0^x f(t)d_{q}t=(1-q)x\sum\limits_{m=0}^\infty q^{m}f(xq^{m}),
\end{equation}
and, more generally, 
\begin{equation*} 
\int\limits_a^b f(x)d_{q}x=\int\limits_0^b f(x)d_{q}x-
\int\limits_0^a f(x)d_{q}x,
\end{equation*}
provided the sums converge absolutely. Note that 
\begin{equation*}
\int\limits_a^x D_qf(t)d_{q}t= f(x) -f(a). 
\end{equation*}
We refer the reader to \cite{CK20,Ern12} for more details about the $q$-calculus.

In the following lemma, we give the $q$-version of the differentiation of functions defined in terms of integrals depending on a parameter.

\begin{lem}\label{lemma1}
    Let $f:\mathbb{R}^2\rightarrow\mathbb{R}$ be a function such that $D_{q,x}f(x,y)$ exists. Then, the function $F$ defined on $\mathbb{R}$ by $F(x)=\int_{0}^{x}f(x,t)\d_{q} t$ is $q$-differentiable and we have
    \begin{equation*}
        D_{q}F(x)=f(qx,x) + \int_{0}^{x}D_{q,x}f(x,t)\d_{q}t,
    \end{equation*}
    or equivalently
    \begin{equation*}
        D_{q}F(x)=f(x,x) + \int_{0}^{qx}D_{q,x}f(x,t)\d_{q}t.
    \end{equation*}    
\end{lem}

\begin{proof}
    The $q$-derivative of $F(x)=\int_{0}^{x}f(x,t)\d_{q} t$ is given by
    \begin{align}\label{formula1}
    D_{q,x}F(x) & = \frac{\int_{0}^{x}f(x,t)\d_{q}t - \int_{0}^{qx}f(qx,t)\d_{q}t}{(1-q)x} \nonumber\\
    & = \int_{0}^{x}\frac{f(x,t) - f(qx,t)}{(1-q)x}\d_{q}t + 
    \frac{\int_{0}^{x}f(qx,t)\d_{q}t - \int_{0}^{qx}f(qx,t)\d_{q}t}{(1-q)x}.
    \end{align}
The integrand in the first term in \eqref{formula1} is equal to $D_{q,x}f(x,t)$ and it easy to see that the second term is equal to $f(qx,x)$, ending the proof.    
\end{proof}

\section{Classical Duhamel's principle}
The following variants of Duhamel's principle for which we provide the proofs, are commonly used. They can be found in \cite{Seb22}. For convenience of the reader, we recall them here.

\subsection{First order Duhamel's principle}
Let us consider the following Cauchy problem for the first order non-homogeneous linear evolution equation
\begin{equation}
    \left\lbrace
    \begin{array}{l}
    u_t(t,x)-Lu(t,x)=f(t,x) ,~~~(t,x)\in\left(0,\infty\right)\times \mathbb{R}^{d},\\
    u(0,x)=u_{0}(x),\,\,\, x\in\mathbb{R}^{d}, \label{Equation Duhamel 2}
    \end{array}
    \right.
\end{equation}
where $L$ is a linear differential operator that includes no time derivatives.

\begin{thm}\label{Thm Duhamel 2}
The solution to the Cauchy problem (\ref{Equation Duhamel 2}) is given by
\begin{equation}
    u(t,x)= w(t,x) + \int_0^t v(t,x;s)\d s,\label{Sol Duhamel 2}
\end{equation}
where $w(t,x)$ is the solution to the homogeneous problem
\begin{equation}
    \left\lbrace
    \begin{array}{l}
    w_t(t,x)-Lw(t,x)=0 ,~~~(t,x)\in\left(0,\infty\right)\times \mathbb{R}^{d},\\
    w(0,x)=u_{0}(x),\,\,\, x\in\mathbb{R}^{d},\label{Homog eqn Duhamel 2}
    \end{array}
    \right.
\end{equation}
and $v(t,x;s)$ solves the auxiliary Cauchy problem
\begin{equation}
    \left\lbrace
    \begin{array}{l}
    v_t(t,x;s)-Lv(t,x;s)=0 ,~~~(t,x)\in\left(s,\infty\right)\times \mathbb{R}^{d},\\
    v(s,x;s)=f(s,x),\,\,\, x\in\mathbb{R}^{d},\label{Aux eqn Duhamel 2}
    \end{array}
    \right.
\end{equation}
where $s$ is a time-like parameter varying over $\left(0,\infty\right)$.
\end{thm}

\begin{proof}
Applying the derivative $\partial_{t}$ to $u$ in (\ref{Sol Duhamel 2}) we get
\begin{equation}
    \partial_t u(t,x)=\partial_t w(t,x) + v(t,x;t) + \int_0^t \partial_t v(t,x;s)\d s.\label{Duhamel proof 2.1}
\end{equation}
We note that $v(t,x;t)=f(t,x)$ by the initial condition in (\ref{Aux eqn Duhamel 2}).
For the spatial component $L$ we have
\begin{equation}
    L u(t,x)=L w(t,x) + \int_0^t L v(t,x;s)\d s,\label{Duhamel proof 2.2}
\end{equation}
since $L$ applies only to the variable $x$. Combining (\ref{Duhamel proof 2.1}) and (\ref{Duhamel proof 2.2}) and using that $w$ and $v$ are the solutions to (\ref{Homog eqn Duhamel 2}) and (\ref{Aux eqn Duhamel 2}) we arrive at
\begin{equation*}
    u_t(t,x) - Lu(t,x)= f(t,x).
\end{equation*}
Observing that $u(0,x)=u_0(x)$ from the initial condition in (\ref{Homog eqn Duhamel 2}) completes the proof.
\end{proof}

\subsection{Second order Duhamel's principle}
Let us consider the Cauchy problem for the second order non-homogeneous linear evolution equation
\begin{equation}
    \left\lbrace
    \begin{array}{l}
    u_{tt}(t,x)-Lu(t,x)=f(t,x) ,~~~(t,x)\in\left(0,\infty\right)\times \mathbb{R}^{d},\\
    u(0,x)=u_{0}(x),\,\,\, u_{t}(0,x)=u_{1}(x),\,\,\, x\in\mathbb{R}^{d}, \label{Equation Duhamel 3}
    \end{array}
    \right.
\end{equation}
for some linear partial differential operator $L$ over the space variable $x$.

\begin{thm}\label{Thm Duhamel 3}
The solution to the Cauchy problem (\ref{Equation Duhamel 3}) is given by
\begin{equation}
    u(t,x)= w(t,x) + \int_0^t v(t,x;s)\d s,\label{Sol Duhamel 3}
\end{equation}
where $w(t,x)$ is the solution to the homogeneous problem
\begin{equation}
    \left\lbrace
    \begin{array}{l}
    w_{tt}(t,x)-Lw(t,x)=0 ,~~~(t,x)\in\left(0,\infty\right)\times \mathbb{R}^{d},\\
    w(0,x)=u_{0}(x),\,\,\, w_{t}(0,x)=u_{1}(x),\,\,\, x\in\mathbb{R}^{d},\label{Homog eqn Duhamel 3}
    \end{array}
    \right.
\end{equation}
and $v(t,x;s)$ solves the auxiliary Cauchy problem
\begin{equation}
    \left\lbrace
    \begin{array}{l}
    v_{tt}(t,x;s)-Lv(t,x;s)=0 ,~~~(t,x)\in\left(s,\infty\right)\times \mathbb{R}^{d},\\
    v(s,x;s)=0,\,\,\, v_{t}(s,x;s)=f(s,x),\,\,\, x\in\mathbb{R}^{d},\label{Aux eqn Duhamel 3}
    \end{array}
    \right.
\end{equation}
where $s$ is a time-like parameter varying over $\left(0,\infty\right)$.
\end{thm}

\begin{proof}
As in Theorem \ref{Thm Duhamel 2}, we apply the components of the operator $\partial_{tt}-L$ separately to $u$ in (\ref{Sol Duhamel 3}). For the spatial component $L$ we simply have
\begin{equation}
    L u(t,x)=L w(t,x) + \int_0^t L v(t,x;s)\d s.\label{Duhamel proof 3.1}
\end{equation}
For the temporal component we get
\begin{equation}
    \partial_t u(t,x)=\partial_t w(t,x) + \int_0^t \partial_t v(t,x;s)\d s,\label{Duhamel proof 3.2}
\end{equation}
where we used the fact that $v(t,x;t)=0$ by the imposed initial condition in (\ref{Aux eqn Duhamel 3}). Differentiating (\ref{Duhamel proof 3.2}) again and noting that $\partial_t v(t,x;t) = f(t,x)$, we arrive at
\begin{equation}
    \partial_{tt} u(t,x)=\partial_{tt} w(t,x) + f(t,x) + \int_0^t \partial_{tt} v(t,x;s)\d s.\label{Duhamel proof 3.3}
\end{equation}
Combining (\ref{Duhamel proof 3.1}) and (\ref{Duhamel proof 3.3}) and using the fact that $w$ and $v$ solve (\ref{Homog eqn Duhamel 3}) and (\ref{Aux eqn Duhamel 3}) respectively, we arrive at
\begin{equation*}
    u_{tt}(t,x) - Lu(t,x)= f(t,x).
\end{equation*}
To conclude the proof, we just observe from (\ref{Sol Duhamel 3}) that $u(0,x)=w(0,x)=u_0(x)$ and from (\ref{Duhamel proof 3.2}) that $u_t(0,x)=\partial
_t w(0,x)=u_1(x)$.
\end{proof}

\subsection{$k$-order Duhamel's principle}
More generally, if we consider the Cauchy problem for the $k$-order non-homogeneous linear evolution equation
\begin{equation}\label{Equation Duhamel general}
        \left\lbrace
        \begin{array}{l}        \partial_{t}^{k}u(t,x) - Lu(t,x) = f(t,x) ,~~~(t,x)\in\left(0,\infty\right)\times \mathbb{R}^{d},\\
        \partial_{t}^{j}u(0,x)=u_{j}(x),\,\text{for}\quad j=0,\cdots,k-1, \quad x\in\mathbb{R}^{d},
        \end{array}
        \right.
\end{equation}
then, Duhamel's principle reads

\begin{thm}\label{Thm Duhamel general}
    The solution to the Cauchy problem \eqref{Equation Duhamel general} has the representation
    \begin{equation*}\label{Sol Duhamel general}
    u(t,x)= w(t,x) + \int_0^t v(t,\tau;\tau)\d \tau,
\end{equation*}
where $w(t,x)$ is the solution to the homogeneous problem
\begin{equation*}\label{Homog eqn Duhamel general}
    \left\lbrace
    \begin{array}{l}
    \partial_{t}^{k}w(t,x) - Lw(t,x)=0 ,~~~(t,x)\in\left(0,\infty\right)\times \mathbb{R}^{d},\\
    \partial_{t}^{j}w(0,x)=u_{j}(x),\,\text{for}\quad j=0,\cdots,k-1, \quad x\in\mathbb{R}^{d},
    \end{array}
    \right.
\end{equation*}
and $v(t,x;\tau)$ solves
\begin{equation*}\label{Aux eqn Duhamel general}
    \left\lbrace
    \begin{array}{l}
    \partial_{t}^{k}v(t,x;\tau) - Lv(t,x;\tau)=0 ,~~~(t,x)\in\left(\tau,\infty\right)\times \mathbb{R}^{d},\\
    \partial_{t}^{j}w(\tau,x;\tau)=0,\,\text{for}\quad j=0,\cdots,k-2,\,\,\, \partial_{t}^{k-1}w(\tau,x;\tau)=f(\tau,x),\,\,\, x\in\mathbb{R}^{d},
    \end{array}
    \right.
\end{equation*}
where $\tau\in \left(0,\infty\right)$.
\end{thm}

\begin{proof}
    The proof follows by induction over $k$.
\end{proof}

\section{Duhamel's principle for systems of coupled equations}
In this section we present the proofs for analogous versions of Duhamel's principle for systems of coupled equations of order $1$, $2$, as well as for systems of mixed orders.
 
\subsection{Systems of coupled evolution equations of first order}
Let us consider the following system of non-homogeneous coupled evolution equations of first order:
\begin{equation}
    \left\lbrace
    \begin{array}{l}
    u_t(t,x) + L_1 u(t,x) + L_2 \theta(t,x) =f(t,x), \quad (t,x)\in\left(0,\infty\right)\times \mathbb{R}^d,\\
    \theta_t(t,x) + L_3 \theta(t,x) + L_4 u(t,x) =g(t,x),\quad (t,x)\in\left(0,\infty\right)\times \mathbb{R}^d,\\
    u(0,x)=u_{0}(x),\,\theta(0,x)=\theta_{0}(x),\quad x\in\mathbb{R}^d, \label{Equation Duhamel 4-syst}
    \end{array}
    \right.
\end{equation}
where $L_1$, $L_2$, $L_3$ and $L_4$ are linear differential operators acting only on the spatial variable $x$. Duhamel's principle reads:

\begin{thm}\label{Thm Duhamel 4-syst}
The solution to the system (\ref{Equation Duhamel 4-syst}) is given by
\begin{equation}\label{Sol Duhamel 4-syst}
\left\lbrace
\begin{array}{l}
    u(t,x)= v(t,x) + \int_0^t w(t,x;s)\d_{q} s,\\
    \theta(t,x)= \gamma(t,x) + \int_0^t \zeta(t,x;s)\d_{q} s,
\end{array}
\right.
\end{equation}
where $v(t,x)$ and $\gamma(t,x)$ are solutions to the system of homogeneous equations associated to \eqref{Equation Duhamel 4-syst}. That is
\begin{equation}\label{Homog eqn Duhamel 4-syst}
    \left\lbrace
    \begin{array}{l}
    v_t(t,x) + L_1 v(t,x) + L_2 \gamma(t,x) =0, \quad (t,x)\in\left(0,\infty\right)\times \mathbb{R}^d,\\
    \gamma_t(t,x) + L_3 \gamma(t,x) + L_4 v(t,x) =0,\quad (t,x)\in\left(0,\infty\right)\times \mathbb{R}^d,\\
    v(0,x)=u_{0}(x),\,\gamma(0,x)=\theta_{0}(x),\quad x\in\mathbb{R}^d,
    \end{array}
    \right.
\end{equation}
and $w(t,x;s)$, $\zeta(t,x;s)$ solve the auxiliary system of equations
\begin{equation}\label{Aux eqn Duhamel 4-syst}
    \left\lbrace
    \begin{array}{l}
    w_t(t,x;s) + L_1 w(t,x;s) + L_2 \zeta(t,x;s) = 0, \quad (t,x)\in\left(s,\infty\right)\times \mathbb{R}^d,\\
    \zeta_t(t,x;s) + L_3 \zeta(t,x;s) + L_4 w(t,x;s) = 0,\quad (t,x)\in\left(s,\infty\right)\times \mathbb{R}^d,\\
w(s,x;s)=f(s,x),\,\zeta(s,x;s)=g(s,x),\quad x\in\mathbb{R}^d,
    \end{array}
    \right.
\end{equation}
where $s$ is a time-like parameter varying over $\left(0,\infty\right)$.
\end{thm}

\begin{proof}
To write the system \eqref{Equation Duhamel 4-syst} in matrix form, we define
\begin{equation*}
    U(t,x):= \begin{pmatrix} u(t,x) \\ \theta(t,x) \end{pmatrix}.
\end{equation*}
Then, \eqref{Equation Duhamel 4-syst} can be rewritten as
\begin{equation}\label{matrix form 1-syst}
    \left\lbrace
    \begin{array}{l}
        \partial_t U(t,x) + L U(t,x) = F(t,x), \quad (t,x)\in\left(0,\infty\right)\times \mathbb{R}^d,\\
        U(0,x)=U_0 (x), \quad x\in\mathbb{R}^d,
    \end{array}
    \right.
\end{equation}
where $L$ is the differential matrix operator, defined by
\begin{equation*}
    L= \begin{pmatrix}
    L_1 & L_2\\
    L_4 & L_3
    \end{pmatrix},
\end{equation*}
and
\begin{equation*}
F(t,x) = \begin{pmatrix} f(t,x) \\ g(t,x) \end{pmatrix},\quad
    U_0(x)= \begin{pmatrix} u_0(x) \\ \theta_0(x) \end{pmatrix},
\end{equation*}
for $(t,x)\in\left(0,\infty\right)\times \mathbb{R}^d$. By applying Theorem \ref{Thm Duhamel 2} to \eqref{matrix form 1-syst}, we get the desired representation for the solution. This completes the proof.
\end{proof}

\subsection{Systems of coupled evolution equations of second order}
Let us consider the following system of non-homogeneous coupled evolution equations of first order:
\begin{equation}
    \left\lbrace
    \begin{array}{l}
    u_{tt}(t,x) + L_1 u(t,x) + L_2 \theta(t,x) =f(t,x), \quad (t,x)\in\left(0,\infty\right)\times \mathbb{R}^d,\\
    \theta_{tt}(t,x) + L_3 \theta(t,x) + L_4 u(t,x) =g(t,x),\quad (t,x)\in\left(0,\infty\right)\times \mathbb{R}^d,\\
    u(0,x)=u_{0}(x),\,u_t(0,x)=u_{1}(x), \quad x\in\mathbb{R}^d,\\
    \theta(0,x)=\theta_{0}(x),\,\theta_t(0,x)=\theta_{1}(x),\quad x\in\mathbb{R}^d, \label{Equation Duhamel 5-syst}
    \end{array}
    \right.
\end{equation}
where $L_1$, $L_2$, $L_3$ and $L_4$ are linear differential operators acting only on the spatial variable $x$. Duhamel's principle in this case, reads as follows:

\begin{thm}\label{Thm Duhamel 5-syst}
The solution to the system (\ref{Equation Duhamel 5-syst}) is given by
\begin{equation}\label{Sol Duhamel 5-syst}
\left\lbrace
\begin{array}{l}
    u(t,x)= v(t,x) + \int_0^t w(t,x;s)\d_{q} s,\\
    \theta(t,x)= \gamma(t,x) + \int_0^t \zeta(t,x;s)\d_{q} s,
\end{array}
\right.
\end{equation}
where $v(t,x)$ and $\gamma(t,x)$ are solutions to the system of homogeneous equations associated to \eqref{Equation Duhamel 5-syst}
\begin{equation}\label{Homog eqn Duhamel 5-syst}
    \left\lbrace
    \begin{array}{l}
    v_{tt}(t,x) + L_1 v(t,x) + L_2 \gamma(t,x) =0, \quad (t,x)\in\left(0,\infty\right)\times \mathbb{R}^d,\\
    \gamma_{tt}(t,x) + L_3 \gamma(t,x) + L_4 v(t,x) =0,\quad (t,x)\in\left(0,\infty\right)\times \mathbb{R}^d,\\
    v(0,x)=u_{0}(x),\,v_t(0,x)=\theta_{0}(x),\quad x\in\mathbb{R}^d,\\
    \gamma(0,x)=u_{0}(x),\,\gamma_t(0,x)=\theta_{0}(x),\quad x\in\mathbb{R}^d,
    \end{array}
    \right.
\end{equation}
and $w(t,x;s)$, $\zeta(t,x;s)$ solve the auxiliary system of equations
\begin{equation}\label{Aux eqn Duhamel 5-syst}
    \left\lbrace
    \begin{array}{l}
    w_{tt}(t,x;s) + L_1 w(t,x;s) + L_2 \zeta(t,x;s) = 0, \quad (t,x)\in\left(s,\infty\right)\times \mathbb{R}^d,\\
    \zeta_{tt}(t,x;s) + L_3 \zeta(t,x;s) + L_4 w(t,x;s) = 0,\quad (t,x)\in\left(s,\infty\right)\times \mathbb{R}^d,\\
    w(s,x;s)=0,\,w_t(s,x;s)=f(s,x),\quad x\in\mathbb{R}^d,\\
    \zeta(s,x;s)=0,\,\zeta_t(s,x;s)=g(s,x),\quad x\in\mathbb{R}^d,
    \end{array}
    \right.
\end{equation}
where $s$ is a time-like parameter varying over $\left(0,\infty\right)$.
\end{thm}

\begin{proof}
    By denoting
    \begin{equation*}
    U(t,x):= \begin{pmatrix} u(t,x) \\
    u_t(t,x) \\
    \theta(t,x) \\
    \theta_t(t,x)
    \end{pmatrix},
    \end{equation*}
    the system of equations \eqref{Equation Duhamel 5-syst} rewrites as follows
    \begin{equation}\label{matrix form 2-syst}
    \left\lbrace
    \begin{array}{l}
        \partial_t U(t,x) + L U(t,x) = F(t,x), \quad (t,x)\in\left(0,\infty\right)\times \mathbb{R}^d,\\
        U(0,x)=U_0 (x), \quad x\in\mathbb{R}^d,
    \end{array}
    \right.
\end{equation}
    where $L$ is the differential matrix operator, defined by
\begin{equation*}
    L= \begin{pmatrix}
    0 & -1 & 0 & 0\\
    L_1 & 0 & L_2 & 0\\
    0 & 0 & 0 & -1\\
    L_4 & 0 & L_3 & 0
    \end{pmatrix},
\end{equation*}
and
\begin{equation*}
    F(t,x)= \begin{pmatrix} 0 \\
    f(t,x) \\
    0\\
    g(t,x)
    \end{pmatrix}, \quad
    U_0(x)= \begin{pmatrix} u_0(x) \\
    u_1(x) \\
    \theta_0(x)\\
    \theta_1(x)
    \end{pmatrix},
\end{equation*}
for $(t,x)\in\left(0,\infty\right)\times \mathbb{R}^d$. The representation for the solution follows by applying Theorem \ref{Thm Duhamel 2}, ending the proof.
\end{proof}

\subsection{Systems of coupled evolution equations of mixed orders}
Let us now consider the following system of coupled evolution equations of mixed orders:
\begin{equation}
    \left\lbrace
    \begin{array}{l}
    u_{tt}(t,x) + L_1 u(t,x) + L_2 \theta(t,x) =f(t,x), \quad (t,x)\in\left(0,\infty\right)\times \mathbb{R}^d,\\
    \theta_{t}(t,x) + L_3 \theta(t,x) + L_4 u(t,x) =g(t,x),\quad (t,x)\in\left(0,\infty\right)\times \mathbb{R}^d,\\
    u(0,x)=u_{0}(x),\,u_t(0,x)=u_{1}(x), \quad x\in\mathbb{R}^d,\\
    \theta(0,x)=\theta_{0}(x),\quad x\in\mathbb{R}^d, \label{Equation Duhamel 6-syst}
    \end{array}
    \right.
\end{equation}
where $L_1$, $L_2$, $L_3$ and $L_4$ are linear differential operators acting only on the spatial variable $x$. Duhamel's principle states that:

\begin{thm}\label{Thm Duhamel 6}
The solution to the system (\ref{Equation Duhamel 6-syst}) is given by
\begin{equation}\label{Sol Duhamel 6-syst}
\left\lbrace
\begin{array}{l}
    u(t,x)= v(t,x) + \int_0^t w(t,x;s)\d_{q} s,\\
    \theta(t,x)= \gamma(t,x) + \int_0^t \zeta(t,x;s)\d_{q} s,
\end{array}
\right.
\end{equation}
where $v(t,x)$ and $\gamma(t,x)$ are solutions to the system of homogeneous equations associated to \eqref{Equation Duhamel 6-syst}
\begin{equation}\label{Homog eqn Duhamel 6-syst}
    \left\lbrace
    \begin{array}{l}
    v_{tt}(t,x) + L_1 v(t,x) + L_2 \gamma(t,x) =0, \quad (t,x)\in\left(0,\infty\right)\times \mathbb{R}^d,\\
    \gamma_{t}(t,x) + L_3 \gamma(t,x) + L_4 v(t,x) =0,\quad (t,x)\in\left(0,\infty\right)\times \mathbb{R}^d,\\
    v(0,x)=u_{0}(x),\,v_t(0,x)=\theta_{0}(x),\quad x\in\mathbb{R}^d,\\
    \gamma(0,x)=\theta_{0}(x),\quad x\in\mathbb{R}^d,
    \end{array}
    \right.
\end{equation}
and $w(t,x;s)$, $\zeta(t,x;s)$ solve the auxiliary system of equations
\begin{equation}\label{Aux eqn Duhamel 6-syst}
    \left\lbrace
    \begin{array}{l}
    w_{tt}(t,x;s) + L_1 w(t,x;s) + L_2 \zeta(t,x;s) = 0, \quad (t,x)\in\left(s,\infty\right)\times \mathbb{R}^d,\\
    \zeta_t(t,x;s) + L_3 \zeta(t,x;s) + L_4 w(t,x;s) = 0,\quad (t,x)\in\left(s,\infty\right)\times \mathbb{R}^d,\\
    w(s,x;s)=0,\,w_t(s,x;s)=f(s,x),\quad x\in\mathbb{R}^d,\\
    \zeta(s,x;s)=g(s,x),\quad x\in\mathbb{R}^d,
    \end{array}
    \right.
\end{equation}
where $s$ is a time-like parameter varying over $\left(0,\infty\right)$.
\end{thm}

\begin{proof}
    The proof follows by writing the system \eqref{Equation Duhamel 6-syst} in the following matrix form
    \begin{equation}\label{matrix form 3-syst}
    \left\lbrace
    \begin{array}{l}
        \partial_t U(t,x) + L U(t,x) = F(t,x), \quad (t,x)\in\left(0,\infty\right)\times \mathbb{R}^d,\\
        U(0,x)=U_0 (x), \quad x\in\mathbb{R}^d,
    \end{array}
    \right.
\end{equation}
    where
    \begin{equation*}
    U(t,x):= \begin{pmatrix} u(t,x) \\
    u_t(t,x) \\
    \theta(t,x)
    \end{pmatrix},
    \end{equation*}
    $L$ is the differential matrix operator, defined by
\begin{equation*}
    L= \begin{pmatrix}
    0 & -1 & 0 \\
    L_1 & 0 & L_2 \\
    L_4 & 0 & L_3
    \end{pmatrix},
\end{equation*}
and
\begin{equation*}
    F(t,x)= \begin{pmatrix} 0 \\
    f(t,x) \\
    g(t,x)
    \end{pmatrix}, \quad
    U_0(x)= \begin{pmatrix} u_0(x) \\
    u_1(x) \\
    \theta_0(x)
    \end{pmatrix},
\end{equation*}
for $(t,x)\in\left(0,\infty\right)\times \mathbb{R}^d$, and by applying Theorem \ref{Thm Duhamel 2}.
\end{proof}

\begin{rem}
    Duhamel's principle was stated for systems of coupled evolution equations of orders $1$ and $2$. This can be clearly extended to systems of equations of any order.
\end{rem}

\section{$q$-analogue of Duhamel's principle}
In this section we give the proof for the $q$-analogue of Duhamel's principle for $q$-evolution equations of order $k\geq 1$.

\subsection{First order Duhamel's principle for $q$-evolution equations}
Let us consider the following Cauchy problem for the first order non-homogeneous linear $q$-evolution equation generated by Jackson's operator:
\begin{equation}
    \left\lbrace
    \begin{array}{l}
    D_{q,t}u(t,x)-L_{q,x}u(t,x)=f(t,x) ,~~~(t,x)\in\left(0,\infty\right)_q\times \mathbb{R}_q,\\
    u(0,x)=u_{0}(x),\,\,\, x\in\mathbb{R}_q, \label{Equation Duhamel 4}
    \end{array}
    \right.
\end{equation}
where $L_{q,x}$ is a linear $q$-differential operator that involves no time derivatives. Then, Duhamel's principle reads:

\begin{thm}\label{Thm Duhamel 4}
The solution to the Cauchy problem (\ref{Equation Duhamel 4}) is given by
\begin{equation}
    u(t,x)= w(t,x) + \int_0^t v(t,x;s)\d_{q} s,\label{Sol Duhamel 4}
\end{equation}
where $w(t,x)$ is the solution to the homogeneous problem
\begin{equation}
    \left\lbrace
    \begin{array}{l}
    D_{q,t}w(t,x)-L_{q,x}w(t,x)=0 ,~~~(t,x)\in\left(0,\infty\right)_q\times \mathbb{R}_q,\\
    w(0,x)=u_{0}(x),\,\,\, x\in\mathbb{R}_q,\label{Homog eqn Duhamel 4}
    \end{array}
    \right.
\end{equation}
and $v(t,x;s)$ solves the auxiliary Cauchy problem
\begin{equation}
    \left\lbrace
    \begin{array}{l}
    D_{q,t}v(t,x;s)-L_{q,x}v(t,x;s)=0 ,~~~(t,x)\in\left(s,\infty\right)_q\times \mathbb{R}_q,\\
    v(qs,x;s)=f(s,x),\,\,\, x\in\mathbb{R}_q,\label{Aux eqn Duhamel 4}
    \end{array}
    \right.
\end{equation}
where $s$ is a time-like parameter varying over $\left(0,\infty\right)_q$.
\end{thm}

\begin{proof}
Applying $D_{q,t}$ to $u$ in (\ref{Sol Duhamel 4}) and using Lemma \ref{lemma1}, we get
\begin{equation}
    D_{q,t} u(t,x)=D_{q,t} w(t,x) + v(t,x;t) + \int_0^t D_{q,t} v(t,x;s)\d_{q}s.\label{Duhamel proof 4.1}
\end{equation}
We note that $v(qt,x;t)=f(t,x)$ by the initial condition in (\ref{Aux eqn Duhamel 4}).
For the spatial component $L_{q,x}$ we simply have
\begin{equation}
    L_{q,x} u(t,x)=L_{q,x} w(t,x) + \int_0^t L_{q,x} v(t,x;s)\d_q s,\label{Duhamel proof 4.2}
\end{equation}
as $L_{q,x}$ applies only to the variable $x$. Combining (\ref{Duhamel proof 4.1}) and (\ref{Duhamel proof 4.2}) and using that $w$ and $v$ are the solutions to (\ref{Homog eqn Duhamel 4}) and (\ref{Aux eqn Duhamel 4}) we arrive at
\begin{equation*}
    D_{q,t}u(t,x) - L_{q,x}u(t,x)= f(t,x).
\end{equation*}
Observing that $u(0,x)=u_0(x)$ from the initial condition in (\ref{Homog eqn Duhamel 4}) completes the proof.
\end{proof}

\subsection{Second order Duhamel's principle for $q$-evolution equations}
Let us consider the Cauchy problem for the second order non-homogeneous linear $q$-evolution equation
\begin{equation}
    \left\lbrace
    \begin{array}{l}
    D_{q,tt}u(t,x)-L_{q,x}u(t,x)=f(t,x) ,~~~(t,x)\in\left(0,\infty\right)_q\times \mathbb{R}_q,\\
    u(0,x)=u_{0}(x),\,\,\, D_{q,t}u(0,x)=u_{1}(x),\,\,\, x\in\mathbb{R}_q, \label{Equation Duhamel 5}
    \end{array}
    \right.
\end{equation}
for some linear $q$-differential operator $L_{q,x}$ over the space variable $x$.

\begin{thm}\label{Thm Duhamel 5}
The solution to the Cauchy problem (\ref{Equation Duhamel 5}) is given by
\begin{equation}
    u(t,x)= w(t,x) + \int_0^t v(t,x;s)\d_{q} s,\label{Sol Duhamel 5}
\end{equation}
where $w(t,x)$ is the solution to the homogeneous problem
\begin{equation}
    \left\lbrace
    \begin{array}{l}
    D_{q,tt}w(t,x)-L_{q,x}w(t,x)=0 ,~~~(t,x)\in\left(0,\infty\right)_q\times \mathbb{R}_q,\\
    w(0,x)=u_{0}(x),\,\,\, D_{q,t}w(0,x)=u_{1}(x),\,\,\, x\in\mathbb{R}_q,\label{Homog eqn Duhamel 5}
    \end{array}
    \right.
\end{equation}
and $v(t,x;s)$ solves the auxiliary Cauchy problem
\begin{equation}
    \left\lbrace
    \begin{array}{l}
    D_{q,tt}v(t,x;s)-L_{q,x}v(t,x;s)=0 ,~~~(t,x)\in\left(s,\infty\right)_q\times \mathbb{R}_q,\\
    v(qs,x;s)=0,\,\,\, D_{q,t}v(qs,x;s)=f(s,x),\,\,\, x\in\mathbb{R}_q,\label{Aux eqn Duhamel 5}
    \end{array}
    \right.
\end{equation}
where $s$ is a time-like parameter varying over $\left(0,\infty\right)_q$.
\end{thm}

\begin{proof}
We argue as in Theorem \ref{Thm Duhamel 4} and we apply the components of the operator $D_{q,tt}-L_{q,x}$ separately to $u$ in (\ref{Sol Duhamel 5}). For the spatial component $L_{q,x}$ we simply have
\begin{equation}
    L_{q,x} u(t,x)=L_{q,x} w(t,x) + \int_0^t L_{q,x} v(t,x;s)\d_q s.\label{Duhamel proof 5.1}
\end{equation}
For the temporal component we get
\begin{equation}
    D_{q,t} u(t,x)=D_{q,t} w(t,x) + \int_0^t D_{q,t} v(t,x;s)\d_q s,\label{Duhamel proof 5.2}
\end{equation}
where we used Lemma \ref{lemma1} and the fact that $v(qt,x;t)=0$ by the initial condition in (\ref{Aux eqn Duhamel 5}). Differentiating (\ref{Duhamel proof 5.2}) again and noting that $D_{q,t} v(qt,x;t) = f(t,x)$, we arrive at
\begin{equation}
    D_{q,tt} u(t,x)=D_{q,tt} w(t,x) + f(t,x) + \int_0^t D_{q,tt} v(t,x;s)\d_q s.\label{Duhamel proof 5.3}
\end{equation}
Combining (\ref{Duhamel proof 5.1}) and (\ref{Duhamel proof 5.3}) and using the fact that $w$ and $v$ solve (\ref{Homog eqn Duhamel 5}) and (\ref{Aux eqn Duhamel 5}) respectively, we arrive at
\begin{equation*}
    D_{q,tt}u(t,x) - L_{q,x}u(t,x)= f(t,x).
\end{equation*}
To conclude the proof, we just observe from (\ref{Sol Duhamel 5}) that $u(0,x)=w(0,x)=u_0(x)$ and from (\ref{Duhamel proof 5.2}) that $D_{t}u(0,x)=D_{t}w(0,x)=u_1(x)$.
\end{proof}

\subsection{$k$-order Duhamel's principle for $q$-evolution equations}
If we consider the following Cauchy problem for the $k$-order non-homogeneous $q$-evolution equation
\begin{equation}\label{Equation Duhamel general k}
    \left\lbrace
    \begin{array}{l}
    D_{q,t}^{(k)}u(t,x)-L_{q,x}u(t,x)=f(t,x) ,~~~(t,x)\in\left(0,\infty\right)_q\times \mathbb{R}_q,\\
    D_{q,t}^{(j)}u(0,x)=u_{j}(x),\,\text{for}\quad j=0,\cdots,k-1,\,\,\, x\in\mathbb{R}_q,
    \end{array}
    \right.
\end{equation}
then, the principle reads as follows:

\begin{thm}\label{Thm Duhamel general k}
    The solution to the Cauchy problem \eqref{Equation Duhamel general k} is given by
    \begin{equation}
    u(t,x)= w(t,x) + \int_0^t v(t,x;s)\d_{q} s,
    \end{equation}
    where $w(t,x)$ is the solution to the homogeneous problem
    \begin{equation}
    \left\lbrace
    \begin{array}{l}
    D_{q,t}^{(k)}w(t,x)-L_{q,x}w(t,x)=0 ,~~~(t,x)\in\left(0,\infty\right)_q\times \mathbb{R}_q,\\
    D_{q,t}^{(j)}w(0,x)=w_{j}(x),\,\text{for}\quad j=0,\cdots,k-1,\,\,\, x\in\mathbb{R}_q,
    \end{array}
    \right.
    \end{equation}
    and $v(t,x;s)$ solves the auxiliary Cauchy problem
    \begin{equation}
    \left\lbrace
    \begin{array}{l}
    D_{q,t}^{(k)}v(t,x;s)-L_{q,x}v(t,x;s)=0 ,~~~(t,x)\in\left(s,\infty\right)_q\times \mathbb{R}_q,\\
    D_{q,t}^{j}v(qs,x;s)=0,\,\text{for}\quad j=0,\cdots,k-2,\,\, D_{q,t}^{(k-1)}v(qs,x;s)=f(s,x),\,\,\, x\in\mathbb{R}_q,
    \end{array}
    \right.
    \end{equation}
    where $s$ is a fixed parameter varying over $\left(0,\infty\right)_q$.
\end{thm}

\begin{proof}
    The proof follows by induction over $k$.
\end{proof}

\begin{rem}
    In Theorem \ref{Thm Duhamel 4}, Theorem \ref{Thm Duhamel 5} and Theorem \ref{Thm Duhamel general k}, the $q$-differential operator $L_{q,x}$ can be either generated by Jackson's operator \eqref{Jack} or Rubin's operator \eqref{Rub}. The results remain true.
\end{rem}

\begin{rem}
    We easily see that if $q\rightarrow 1^-$, then we recapture the classical Duhamel's principle.
\end{rem}


\begin{thebibliography}{99}

\bibitem[AHH24]{AHH24} N. Allouch, S. Hamani, J. Henderson. \newblock Boundary value problem for fractional q-difference equations. \newblock {\it Nonlinear Dynamics and Systems Theory}, 24(2) (2024) 111–122.

\bibitem[CK20]{CK20} P. Cheung, V. Kac. \newblock Quantum calculus. \newblock {\it Edwards Brothers. Inc., Ann Arbor, MI, USA}, 2020.

\bibitem[ER18]{ER18} M. R. Ebert, M. Reissig. \newblock Methods for Partial Differential Equations. \newblock {\it Birkhäuser}, 2018.

\bibitem[Ern12]{Ern12} T. Ernst. \newblock A comprehensive treatment of q-calculus. \newblock {\it Birkhäuser/Springer Basel AG, Basel}, 2012.

\bibitem[Eva98]{Eva98} L. C. Evans. \newblock Partial Differential Equations. \newblock {\it American Mathematical Society}, 1998.

\bibitem[GR04]{GR04} G. Gasper, M. Rahman. \newblock Basic Hypergeometric Series. \newblock {\it 2nd ed. Cambridge University Press}, 2004.

\bibitem[Ism05]{Ism05} M. E. H. Ismail. \newblock Classical and Quantum Orthogonal Polynomials in One Variable. \newblock {\it Cambridge University Press}, 2005.

\bibitem[Jac10]{Jac10} F. H. Jackson. \newblock On a q-definite integrals. \newblock {\it Quart. J. Pure and Appl. Math.}, 41, 193–203. 1910.

\bibitem[KR24]{KR24} J. Y. Kang, C. S. Ryoo, \newblock Forms of (q,h)-difference equations and properties of their solutions. \newblock {\it AIMS Mathematics}, 9(11): 29645–29661. 2024.

\bibitem[Rub97]{Rubin1} R. L. Rubin. \newblock A $q^2$-analogue operator for $q^2$-analogue Fourier analysis. \newblock {\it J. Math. Anal. Appl.}, 212, 571-582. 1997.

\bibitem[Rub07]{Rubin2} R. L. Rubin. \newblock Duhamel solutions of non-homogeneous $q^2$-analogue wave equations. \newblock {\it Proc. Amer. Math. Soc.}, 135, 777-785. 2007.

\bibitem[Seb22]{Seb22} M. E. Sebih. Evolution equations with singular coefficients. PhD thesis. \newblock {\it Université Djillali Liabes}, 2022.

\bibitem[SSK24]{SSK24} B. Shimelis, D. L. Suthar, D. Kumar. \newblock Fractional q-calculus operators and q-kinetic equations. \newblock {\it Axioms}, 13(2), 78. 2024.

\bibitem[ZG24]{ZG24} A. Zafer, Z. N. Gurkan. \newblock Oscillation behavior of second-order self-adjoint q-difference equations. \newblock {\it AIMS Mathematics}, 9(7): 16876–16884. 2024.

\end{thebibliography}
\end{document}